\documentclass[leqno,12pt]{amsart} 
\usepackage[top=30truemm,bottom=30truemm,left=25truemm,right=25truemm]{geometry}
\usepackage{amssymb}
\usepackage{amsmath}
\usepackage{amsthm}
\usepackage{amscd}
\usepackage{mathrsfs}
\usepackage{graphicx}
\usepackage{color}
\usepackage{url}
\usepackage{enumitem}
  \makeatletter

\@addtoreset{equation}{section}
\@namedef{subjclassname@2020}{\textup{2020} Mathematics Subject Classification}
\makeatother
\theoremstyle{plain} 
\newtheorem{theorem}{\indent\bf Theorem}[section]
\newtheorem{lemma}[theorem]{\indent\bf Lemma}
\newtheorem{corollary}[theorem]{\indent\bf Corollary}

\theoremstyle{definition} 
\newtheorem{definition}[theorem]{\indent\bf Definition}
\newtheorem{remark}[theorem]{\indent\bf Remark}

\newcommand{\ddbar}{\partial \overline{\partial}}

\newcommand{\ai}{\sqrt{-1}}
\newcommand{\llangle}{\langle\!\langle}
\newcommand{\rrangle}{\rangle\!\rangle}

\newcommand{\R}{\mathbb{R}}

\newcommand{\C}{\mathbb{C}}

\newcommand{\Z}{\mathbb{Z}}
\newcommand{\B}{\mathbb{B}}

\newcommand{\Id}{\mathrm{Id}}

\newcommand{\sasyugo}{\setminus\!}

\begin{document}
\pagestyle{plain}
\thispagestyle{plain}

\title[Asymptotic expansions of $L^2$-extension indices and curvature positivity]
{Asymptotic expansions of $L^2$-extension indices and curvature positivity}

\author[T. INAYAMA]{Takahiro INAYAMA}
\address{Department of Mathematics\\
	College of Science\\
	Rikkyo University\\
	3-34-1, Nishi-Ikebukuro, Toshima-ku\\
	Tokyo, 171-8501\\
	Japan
}
\email{5073974@rikkyo.ac.jp}
\email{inayama570@gmail.com}
\subjclass[2020]{32A36, 32U05}
\keywords{ 
	Ohsawa--Takegoshi extension theorem, $L^2$-extension index, Griffiths positivity, $q$-positivity%
}

\begin{abstract}
	In this paper, we prove an asymptotic expansion of the $L^2$-extension index of a smooth Hermitian metric on a holomorphic vector bundle, in which the Chern curvature appears as the second-order coefficient.
	By using this expansion, we show in a unified way that there is an equivalence between how sharp the $L^2$-extension is and how positive or negative the curvature is.
	We also introduce a new notion of $q$-$L^2$-extension indices and investigate partial positivity and flatness in terms of these indices.
\end{abstract}


\maketitle
\setcounter{tocdepth}{2}

\section{Introduction}\label{sec:intro}

The Ohsawa--Takegoshi $L^2$-extension theorem \cite{OT87} is a fundamental theorem concerning the $L^2$-extension of holomorphic functions.
After B\l ocki \cite{Blo13} and Guan--Zhou \cite{GZ15} established the optimal Ohsawa--Takegoshi $L^2$-extension theorem, the notion called the minimal extension property or the optimal $L^2$-extension property has been studied in various contexts (see \cite{HPS18}, \cite{DNW21}, \cite{DNWZ23}).
In many settings, this property is known to be equivalent to the semipositivity of the curvature of a metric.
It is also known that, if we allow the constant in the $L^2$-estimate of the Ohsawa--Takegoshi extension theorem to depend on the weight, this estimate can be sharpened (see \cite{Hos19}, \cite{Kik22}).
In light of these results, it is natural to ask how the sharpness of the $L^2$-extension is related to the positivity of the curvature.
To study this question quantitatively, the author introduced the notion of $L^2$-extension indices in \cite{Ina26}.
The $L^2$-extension index is a function that gives the minimal constant with respect to the $L^2$-estimate of an Ohsawa--Takegoshi-type extension at each point.
In \cite{Ina26}, it was proved that there is an equivalence between how sharp the $L^2$-extension is and how positive the curvature is.
Subsequently, related characterizations of Griffiths positivity, pluriharmonicity and flatness were obtained by Liu--Xu \cite{LX24}.

The aim of this paper is to show that these characterizations all stem from a single asymptotic expansion: the $L^2$-extension index admits a second-order expansion whose coefficients are given by the Chern curvature.
Once the expansion is established, characterizations of positivity, negativity and flatness of the curvature follow immediately and in a unified way.

First, we recall the definition of $L^2$-extension indices for smooth Hermitian vector bundles.

\begin{definition}[{\cite[(A.1)]{Ina26}}]\label{def:l2indexvector}
	Let $\pi:E\to \Omega$ be a holomorphic vector bundle over $\Omega\subset \C^n$ and $h$ be a smooth Hermitian metric on $E$.
	We set $\Omega_P= \{ (a,r,s,A)\in \Omega \times \R_{>0}\times \R_{>0}\times\mathbf{U}(n) \mid a + P_{r,s,A}\subset \Omega \}$, where $\mathbf{U}(n)$ is the unitary group of degree $n$ and $P_{r,s,A}$ is a holomorphic cylinder defined by $P_{r,s,A} = A(\Delta_r\times \B^{n-1}_s)$ for $r,s>0$.
	We define {\it the $L^2$-extension index} $L_h$ of $h$ on $\Omega_P\times_\Omega E\sasyugo \{0\} = \{(a,r,s,A,\xi)\in \Omega_P\times E\sasyugo \{0\}\mid a = \pi(\xi) \}$ by
	$$
		L_h(a,r,s,A,\xi)=\inf\left\{ \frac{\int_{a+P_{r,s,A}}|s|^2_h}{|P_{r,s,A}| |\xi|^2_{h(a)}} ~\middle|~ s\in A^2(a+P_{r,s,A},E; h), s(a)=\xi \right\}.
	$$
	Here $A^2(a+P_{r,s,A},E; h)$ denotes the space of square integrable holomorphic sections of $E$ on $a+P_{r,s,A}$ with respect to $h$.
\end{definition}

The main theorem of this paper is the following expansion formula.

\begin{theorem}\label{mainthm:Griffithsexpansion}
	For fixed $a\in \Omega$ and $\xi \in E_a\sasyugo \{0\}$, we have the following expansion
	$$
		L_{h}(a,r,s,A, \xi) = 1 - \frac{r^2}{2} \frac{\langle \Theta_h(Ae_1, \overline{Ae_1})\xi, \xi\rangle_{h(a)}}{|\xi|^2_{h(a)}} - \frac{s^2}{n} \sum_{j=2}^n \frac{\langle \Theta_h(Ae_j, \overline{Ae_j})\xi, \xi \rangle_{h(a)}}{|\xi|^2_{h(a)}} + O((r^2+s^2)^2),
	$$
	where $\{ e_1,\dots,e_n\}$ denotes the standard basis of $\mathbb{C}^n$.
\end{theorem}

Throughout this paper, we basically give results when $n\geq 2$.
However, all of them are valid in the $n=1$ case as well if we omit $s$ and $A$, change $P_{r,s,A}$ to $\Delta_r$ and $O((r^2+s^2)^2)$ to $O(r^4)$, and add some modifications if necessary.
We omit to explicitly state the theorems in the $n=1$ case to make them simple.

The point of Theorem \ref{mainthm:Griffithsexpansion} is that the Chern curvature appears exactly as the second-order coefficient of the $L^2$-extension index.
As a direct consequence, we obtain an equivalence between two-sided bounds of the curvature in the sense of Griffiths and two-sided asymptotic estimates of the index (Theorem \ref{thm:equivalence}), which generalizes \cite[Theorem 1.4]{Ina26} and \cite[Theorem 1.1]{LX24}.
In particular, not only the positivity but also the negativity of the curvature can be read off from the index.
Combining the expansion with Berndtsson's curvature formula \cite{Ber11}, we also obtain an expansion formula for the natural $L^2$-metric on a direct image sheaf (Corollary \ref{cor:berndtsson-directimage}).

We further introduce the $q$-$L^2$-extension index $L^{(q)}_h$ and prove the corresponding expansion (Theorem \ref{thm:qexpansion}).
This index can give a characterization of eigenvalue-sum conditions arising in uniform partial positivity in terms of $L^{(q)}_h$ (Theorem \ref{thm:qcharacterization}).
We also obtain a partial positivity version of the minimal/optimal $L^2$-extension property (Theorem \ref{thm:optimalqpositivity})
and a characterization of the flatness of a Hermitian metric by the condition $L^{(q)}_h \equiv 1$ for $1\leq q \leq n-1$ (Theorem \ref{thm:flat-q}).

The organization of the paper is as follows.
In Section \ref{sec:l2ext}, we prove Theorem \ref{mainthm:Griffithsexpansion} and discuss its applications.
In Section \ref{sec:qpositivity-flatness}, we introduce the $q$-$L^2$-extension index and study partial positivity and flatness.
In Appendix \ref{appendix}, we prove a lemma on holomorphic frames used in the proof of the main theorem.

\subsection*{Acknowledgment}\label{subsec-ack}
The author would like to thank Wang Xu for helpful comments.
He is supported by Grant-in-Aid for Early-Career Scientists $\sharp$23K12978 from Japan Society for the Promotion of Science (JSPS).

\section{$L^2$-extension}\label{sec:l2ext}

In this section, we give a proof of Theorem \ref{mainthm:Griffithsexpansion} and study its applications.

\begin{proof}[Proof of Theorem \ref{mainthm:Griffithsexpansion}]\label{proof:Griffithsexpansion}

	Put $x=z-a$, and choose once and for all a holomorphic frame
	around $a$, as in Lemma \ref{lem:frame}, such that
	\[
		h(a+x)
		=
		I-\left(\sum_{p,q}
		c_{p\bar q\lambda\bar\mu}x_p\bar x_q\right)_{\lambda,\mu}
		+Q^{2,1}(x)+Q^{1,2}(x)+ O(|x|^4),
	\]
	where $Q^{2,1}$ and $Q^{1,2}$ have bidegrees $(2,1)$ and
	$(1,2)$, respectively.
	For $A\in\mathbf U(n)$, consider the change of coordinates
	\[
		z=a+Aw=:\varphi(w), \hspace{5mm} \widetilde{h}(w):= \varphi^*h(w)=h(a+Aw)
	\]
	and pull back the above fixed frame.
	Since $A$ is unitary, $L_h(a,r,s,A,\xi)=L_{\widetilde{h}}(0,r,s,I,\xi)$.
	Then we get
	\begin{align*}
		\widetilde{h} = {} & {} I - \left( \sum_{j,k}\widetilde{c}_{j\bar{k}\lambda\bar{\mu}}w_j\bar{w}_k \right)_{\lambda,\mu}                                                                                                                                                \\
		                   & + \left( \sum_{j_1, j_2, k}q^{j_1j_2\bar{k}}_{\lambda\bar{\mu}}w_{j_1}w_{j_2}\bar{w}_k \right)_{\lambda, \mu} + \left( \sum_{j,k_1,k_2}r^{j\bar{k}_1\bar{k}_2}_{\lambda\bar{\mu}}w_j\bar{w}_{k_1}\bar{w}_{k_2} \right)_{\lambda, \mu} + O(|w|^4).
	\end{align*}
	The coefficients may depend on $A$, but again since $A$ is unitary, they are bounded uniformly in $A$.
	We take $\rho>0 $ such that
	\[
		(1-C_0|w|^2)I
		\leq \widetilde h(w)
		\leq (1+C_0|w|^2)I
	\]
	for $|w|<\rho$, with $C_0$ and $\rho$ independent of
	$A\in\mathbf U(n)$ and satisfying $C_0 \rho^2 < 1/2$.
	Here, the inequality $A\leq B$ for two Hermitian matrices $A$ and $B$ means that $B-A$ is positive semidefinite.
	After decreasing $\rho$ if necessary,
	we may also assume that the absolute value of each entry of $O(|w|^4)$ is bounded by $C|w|^4$ for some $C>0$, which is again uniform in $A\in\mathbf U(n)$.
	Afterwards, we take $r, s>0$ satisfying $r^2 + s^2 < \rho^2$.

	Let $\mathcal{H}_0 = \{ s\in A^2(P_{r,s,I},E; \widetilde{h}) \mid s(0) = 0 \}$.
	By standard Hilbert space theory, there exists a unique element $s_{r,s}\in A^2(P_{r,s,I},E; \widetilde{h})$ satisfying $s_{r,s}(0)= \xi$ and
	$$
		L_{\widetilde{h}}(0,r,s,I,\xi) = \frac{\int_{P_{r,s,I}} |s_{r,s}|^2_{\widetilde{h}}}{|P_{r,s,I}| |\xi|^2_{\widetilde{h}(0)}},
	$$
	where $s_{r,s}$ is orthogonal to $\mathcal{H}_0$.
	For $\widetilde{h}$, we write $\llangle f, g\rrangle_{\widetilde{h}}:= \int_{P_{r,s,I}} \langle f, g \rangle_{\widetilde{h}}$, $\langle f, g \rangle_{\widetilde{h}} = {}^t f \widetilde{h} \bar{g},$ and $s_{r,s} = \sum_{|\alpha|=0}^\infty \mathbf{a}_{\alpha, (r,s)} w^\alpha$, where each $\alpha=(\alpha_1, \ldots, \alpha_n)$ is a multi-index with $|\alpha|= \alpha_1 + \cdots + \alpha_n$ and $\mathbf{a}_{\alpha, (r,s)} \in \mathbb{C}^r$.
	Note that $s_{r,s}$, and therefore each coefficient of $s_{r,s}$, may depend on $r$ and $s$, and $\mathbf{a}_{0, (r,s)} = \xi$.
	From what we have discussed above, we obtain $\llangle s_{r,s}, s_{r,s} \rrangle_{\widetilde{h}} = \llangle s_{r,s}, \xi \rrangle_{\widetilde{h}}$ and
	\begin{align*}
		\llangle s_{r,s}, \xi \rrangle_{\widetilde{h}} & = \sum_{|\alpha|=0}^\infty \int_{P_{r,s,I}} {}^t \mathbf{a}_{\alpha, (r,s)} \widetilde{h} \bar{\xi} w^\alpha                                                                                                                                         \\
		                                               & = \sum_{|\alpha|=0}^\infty \int_{P_{r,s,I}} {}^t \mathbf{a}_{\alpha, (r,s)} \Bigg( I - \left( \sum_{j,k}\widetilde{c}_{j\bar{k}\lambda\bar{\mu}}w_j\bar{w}_k \right)
		+ \left( \sum_{j_1, j_2, k}q^{j_1j_2\bar{k}}_{\lambda\bar{\mu}}w_{j_1}w_{j_2}\bar{w}_k \right)                                                                                                                                                                                                        \\
		                                               & \phantom{{}={} \sum_{|\alpha|=0}^\infty \int_{P_{r,s,I}} {}^t \mathbf{a}_{\alpha, (r,s)} \Bigg( I {}} + \left( \sum_{j,k_1,k_2}r^{j\bar{k}_1\bar{k}_2}_{\lambda\bar{\mu}}w_j\bar{w}_{k_1}\bar{w}_{k_2} \right) + O(|w|^4) \Bigg) \bar{\xi} w^\alpha.
	\end{align*}
	Set
	\[
		I_{\alpha\beta}:= \int_{P_{r,s,I}} w^\alpha \bar{w}^\beta, \hspace{5mm} \mathcal{V}_{\alpha}(r,s):= \int_{P_{r,s,I}} |w^\alpha|^{2}.
	\]
	Let us consider the action $R_\theta: (w_1, \ldots, w_n) \mapsto (e^{\ai\theta_1}w_1, \ldots, e^{\ai\theta_n}w_n)$ for $\theta=(\theta_1, \ldots, \theta_n)\in \mathbb{T}^n$.
	Since $P_{r,s,I}$ is invariant under $R_\theta$ for $\theta\in \mathbb{T}^n$, we get
	\[
		I_{\alpha\beta} = e^{\ai\sum_i (\alpha_i-\beta_i)\theta_i} I_{\alpha\beta}.
	\]
	Hence, $I_{\alpha\beta} = 0$ if $\alpha \neq \beta$ and $I_{\alpha\alpha} = \mathcal{V}_\alpha(r,s)$.

	We compute each term on the right-hand side of the above expression of $\llangle s_{r,s}, \xi \rrangle_{\widetilde{h}}$.
	First we get
	\begin{equation}\label{eq:firstterm}
		\sum_{|\alpha|=0}^\infty \int_{P_{r,s,I}} {}^t \mathbf{a}_{\alpha, (r,s)} \bar{\xi} w^\alpha = \int_{P_{r,s,I}} {}^t \xi \bar{\xi} = |P_{r,s,I}| |\xi|^2.
	\end{equation}
	Next, we obtain
	\begin{align}
		\sum_{|\alpha|=0}^\infty \int_{P_{r,s,I}} {}^t \mathbf{a}_{\alpha, (r,s)} \left(\sum_{j,k}\widetilde{c}_{j\bar{k}\lambda\bar{\mu}}w_j\bar{w}_k\right)_{\lambda,\mu}  \bar{\xi} w^\alpha & = \sum_j \int_{P_{r,s,I}} {}^t \mathbf{a}_{0, (r,s)} \left(\widetilde{c}_{j\bar{j}\lambda\bar{\mu}}\right)_{\lambda,\mu} \bar{\xi}|w_j|^2           \\
		                                                                                                                                                                                        & =\sum_j \left( \mathcal{V}_{e_j} \sum_{\lambda,\mu} \widetilde{c}_{j\bar{j}\lambda\bar{\mu}} \xi_\lambda\bar{\xi}_\mu \right)\label{eq:secondterm}.
	\end{align}
	Here, $e_j$ is the multi-index with $1$ at the $j$-th entry and $0$ at the other entries, which means that
	$$
		\mathcal{V}_{e_j} = \int_{P_{r,s,I}} |w_j|^{2}.
	$$
	It is also clear that
	\begin{equation}\label{eq:thirdterm}
		\sum_{|\alpha|=0}^\infty \int_{P_{r,s,I}} {}^t \mathbf{a}_{\alpha, (r,s)} \left( \sum_{j_1, j_2, k}q^{j_1j_2\bar{k}}_{\lambda\bar{\mu}}w_{j_1}w_{j_2}\bar{w}_k \right)_{\lambda,\mu} \bar{\xi} w^{\alpha} =0.
	\end{equation}
	Then we compute the fourth term
	\[
		\int_{P_{r,s,I}} {}^t s_{r,s} \left( \sum_{j,k_1,k_2}r^{j\bar{k}_1\bar{k}_2}_{\lambda\bar{\mu}}w_j\bar{w}_{k_1}\bar{w}_{k_2} \right)_{\lambda,\mu} \bar{\xi}
	\]
	We set $R = \left( \sum_{j,k_1,k_2}r^{j\bar{k}_1\bar{k}_2}_{\lambda\bar{\mu}}w_j\bar{w}_{k_1}\bar{w}_{k_2} \right)_{\lambda,\mu}$ and $u_{r,s}:= s_{r,s} - \xi$.
	Since $R$ is of bidegree $(1,2)$, it follows that
	\begin{equation}\label{eq:bidegree12}
		\int_{P_{r,s,I}}{}^t s_{r,s} R \bar{\xi} = \int_{P_{r,s,I}}{}^t u_{r,s} R \bar{\xi}.
	\end{equation}
	The Cauchy--Schwarz inequality implies that
	\begin{equation}\label{eq:CS}
		\left| \int_{P_{r,s,I}} {}^tu_{r,s}R\bar{\xi} \right| \leq \left( \int_{P_{r,s,I}} |u_{r,s}|^2_{I} \right)^{1/2} \left( \int_{P_{r,s,I}} |R\bar{\xi}|^2_{I} \right)^{1/2}.
	\end{equation}
	From Lemma \ref{lem:deviation} below, we have
	\begin{equation}\label{eq:deviation}
		\int_{P_{r,s,I}} |u_{r,s}|^2_I \leq 4C_0^2 (r^2+s^2)^2 |P_{r,s,I}| |\xi|_I^2.
	\end{equation}
	It follows from the fact that $|R\bar{\xi}|\leq C' |w|^3|\xi|_I$ for some $C'>0$ that
	\begin{equation}\label{eq:Rxi}
		\int_{P_{r,s,I}} |R\bar{\xi}|^2_I \leq C'^2 |\xi|^2_I |P_{r,s,I}| (r^2+s^2)^3.
	\end{equation}
	Combining (\ref{eq:bidegree12}), (\ref{eq:CS}), (\ref{eq:deviation}) and (\ref{eq:Rxi}), we obtain
	\begin{equation}\label{eq:fourthterm}
		\left| \int_{P_{r,s,I}} {}^t s_{r,s} R \bar{\xi} \right| \leq C'' |\xi|^2_I |P_{r,s,I}| (r^2+s^2)^{5/2}
	\end{equation}
	for some $C''>0$.

	Finally, we can estimate the fifth term
	\[
		\int_{P_{r,s,I}} {}^t s_{r,s} O(|w|^4) \bar{\xi}.
	\]
	Using the Cauchy--Schwarz inequality again, we have
	\begin{align*}
		\left| \int_{P_{r,s,I}} {}^ts_{r,s} O(|w|^4) \bar{\xi}  \right| & \leq \int_{P_{r,s,I}} \left|  {}^ts_{r,s} O(|w|^4) \bar{\xi} \right|                                                     \\
		                                                                & \leq \left( \int_{P_{r,s,I}} |s_{r,s}|^2_I \right)^{1/2} \left( \int_{P_{r,s,I}} |O(|w|^4) \bar{\xi}|^2_I \right)^{1/2}.
	\end{align*}
	The minimality of $s_{r,s}$ implies that
	\begin{align*}
		\int_{P_{r,s,I}} |s_{r,s}|^2_I & \leq \frac{1}{1-C_0(r^2+s^2)} \int_{P_{r,s,I}} |s_{r,s}|^2_{\widetilde{h}} \\
		                               & \leq \frac{1}{1-C_0(r^2+s^2)} \int_{P_{r,s,I}} |\xi|^2_{\widetilde{h}}     \\
		                               & \leq \frac{1 + C_0 (r^2+s^2)}{1-C_0(r^2+s^2)} \int_{P_{r,s,I}} |\xi|^2_I   \\
		                               & = \frac{1 + C_0 (r^2+s^2)}{1-C_0(r^2+s^2)} |\xi|^2_I |P_{r,s,I}|.
	\end{align*}
	We can also see that
	\begin{equation*}
		\int_{P_{r,s,I}} |O(|w|^4) \bar{\xi}|^2_I  \leq C''' |\xi|^2_I |P_{r,s,I}| (r^2+s^2)^4.
	\end{equation*}
	Hence, we have that
	\begin{equation}\label{eq:fifthterm}
		\left| \int_{P_{r,s,I}} {}^ts_{r,s} O(|w|^4) \bar{\xi}  \right| \leq C'''' |\xi|^2_I |P_{r,s,I}| (r^2+s^2)^2.
	\end{equation}

	Combining (\ref{eq:firstterm}), (\ref{eq:secondterm}), (\ref{eq:thirdterm}), (\ref{eq:fourthterm}) and (\ref{eq:fifthterm}), we finally get
	\begin{align*}
		L_{\widetilde{h}}(0,r,s,I,\xi) = \frac{\int_{P_{r,s,I}}|s_{r,s}|^2_{\widetilde{h}}}{|P_{r,s,I}||\xi|^2_{\widetilde{h}(0)}} = 1- \frac{\sum_j \left( \mathcal{V}_{e_j} \sum_{\lambda,\mu} \widetilde{c}_{j\bar{j}\lambda\bar{\mu}} \xi_\lambda\bar{\xi}_\mu \right)}{|P_{r,s,I}||\xi|^2_{\widetilde{h}(0)}} + O((r^2+s^2)^2).
	\end{align*}
	Based on this, we convert this result into the expansion of $L_h(a,r,s,A,\xi)$ to complete the proof.
	Simple computation shows that
	\begin{equation*}
		\frac{\mathcal{V}_{e_1}}{|P_{r,s,I}|} = \frac{r^2}{2}, \hspace{10mm} \frac{\mathcal{V}_{e_\alpha}}{|P_{r,s,I}|} = \frac{s^2}{n}~(2\leq \alpha \leq n).
	\end{equation*}
	If we write $\Theta_h = \sum_{p,q}\Theta_{h, p\bar{q}}dz_p\wedge d\bar{z}_q$ and $\Theta_{\widetilde{h}} = \sum_{j,k}\Theta_{\widetilde{h}, j\bar{k}}dw_j\wedge d\bar{w}_k$, we have
	\begin{equation*}
		\Theta_{\widetilde{h}, j\bar{k}}(0) = \sum_{p,q} \Theta_{h, p\bar{q}}(a)A_{pj}\overline{A}_{qk},
	\end{equation*}
	where $A_{pj}$ is the $(p,j)$-entry of $A$.
	Using the identity
	\begin{equation*}
		\sum_{\lambda,\mu} \widetilde{c}_{j\bar{j}\lambda\bar{\mu}} \xi_\lambda\bar{\xi}_\mu = \langle \Theta_{\widetilde{h}, j\bar{j}}\xi, \xi \rangle_{\widetilde{h}(0)},
	\end{equation*}
	we obtain
	\begin{equation*}
		\langle \Theta_{\widetilde{h}, j\bar{j}}\xi, \xi \rangle_{\widetilde{h}(0)} = \langle \Theta_{h}(Ae_j, \overline{Ae_j})\xi, \xi \rangle_{h(a)}.
	\end{equation*}
	Hence, we can get
	\begin{equation*}
		L_{h}(a,r,s,A, \xi) = 1 - \frac{r^2}{2} \frac{\langle \Theta_h(Ae_1, \overline{Ae_1})\xi, \xi\rangle_{h(a)}}{|\xi|^2_{h(a)}} - \frac{s^2}{n} \sum_{j=2}^n \frac{\langle \Theta_h(Ae_j, \overline{Ae_j})\xi, \xi \rangle_{h(a)}}{|\xi|^2_{h(a)}} + O((r^2+s^2)^2).
	\end{equation*}
	Note that the estimate of the $O((r^2+s^2)^2)$ term is uniform in $A\in \mathbf{U}(n)$ and $\xi$.
\end{proof}

\begin{lemma}\label{lem:deviation}
	Let the notation be the same as above.
	If $s$ is the minimal-norm extension of $\xi$ on $P_{r,s,I}$ and $u:= s- \xi$, we have
	\[
		\int_{P_{r,s,I}} |u|^2_I \leq 4C_0^2 (r^2+s^2)^2 |P_{r,s,I}| |\xi|_I^2.
	\]
\end{lemma}

\begin{proof}[Proof of Lemma \ref{lem:deviation}]\label{proof:deviation}
	Since $u(0)=0$ and $s$ has minimal norm, $s$ is orthogonal to $u$ with respect to $\widetilde{h}$, that is,
	\[
		\llangle s, u\rrangle_{\widetilde{h}} = 0.
	\]
	Hence, writing $\widetilde{h}=I+(\widetilde{h}-I)$, we have
	\[
		\llangle u, u\rrangle_{\widetilde{h}} = - \llangle \xi, u\rrangle_{\widetilde{h}} = - \int_{P_{r,s,I}} {}^t \xi \bar{u} - \int_{P_{r,s,I}} {}^t \xi (\widetilde{h}-I) \bar{u}.
	\]
	The first term is zero due to the same reason as above.
	The Cauchy--Schwarz inequality gives the second term as
	\[
		\left| \int_{P_{r,s,I}} {}^t \xi (\widetilde{h}-I) \bar{u} \right| \leq C_0 |\xi|_I (r^2+s^2) \int_{P_{r,s,I}} |u|_I \leq C_0 |\xi|_I (r^2+s^2) |P_{r,s,I}|^{1/2} \left( \int_{P_{r,s,I}} |u|_I^2 \right)^{1/2}.
	\]
	On the other hand, since $\widetilde{h}\geq\left( 1-C_0(r^2+s^2) \right)I$ on $P_{r,s,I}$, we have
	\[
		\llangle u, u\rrangle_{\widetilde{h}} \geq \left( 1-C_0(r^2+s^2) \right) \int_{P_{r,s,I}} |u|_I^2.
	\]
	Combining these two inequalities, we get
	\[
		\left( \int_{P_{r,s,I}} |u|_I^2 \right)^{1/2} \leq \frac{C_0|\xi|_I}{1-C_0(r^2+s^2)}(r^2+s^2)|P_{r,s,I}|^{1/2},
	\]
	which implies the desired inequality since $\frac{1}{1-C_0(r^2+s^2)} < 2$.
\end{proof}

Note that this lemma is true whenever a domain is invariant under the action of $\mathbb{T}^n$ such as $\B(0;r), \B^k_r\times \B^{n-k}_s$, etc.

As a corollary of Theorem \ref{mainthm:Griffithsexpansion},
for example, we can immediately see that
\begin{equation*}
	\lim_{r, s\to +0} L_h(a,r,s,A,\xi) = 1,
\end{equation*}
(see Proposition 2.6 and 2.10 in \cite{Ina26}).
As another corollary,
we can get the following result, which is a generalization of \cite[Theorem 1.4]{Ina26} and \cite[Theorem 1.1]{LX24}, in a more organized way.

\begin{theorem}\label{thm:equivalence}
	Let $\omega = \ai \sum_j dz_j \wedge d\bar{z}_j$ be the standard K\"ahler form on $\Omega$, $\pi:E\to \Omega$ be a holomorphic vector bundle over $\Omega$ and $h$ be a smooth Hermitian metric on $E$.
	Fix $a\in\Omega$ and $c,d\in\mathbb R$ with $c\leq d$.
	Then the following two conditions are equivalent:
	\begin{enumerate}
		\item $c\omega \otimes  \Id_{E} \leq_{\mathrm{Grif.}}  \ai\Theta_{h} \leq_{\mathrm{Grif.}} d\omega \otimes \Id_{E}$ at $a$,
		\item for any $\varepsilon>0$, there exists $\rho_\varepsilon > 0$ such that
		      \[
			      1 - (d+\varepsilon)\mathfrak{d}(P_{r,s,A})^2 \leq L_h(a,r,s,A,\xi) \leq 1 - (c-\varepsilon)\mathfrak{d}(P_{r,s,A})^2
		      \]
		      for any $r,s>0$ with $r^2+s^2 < \rho_\varepsilon^2$, $A\in \mathbf{U}(n)$ and $\xi \in E_a \sasyugo \{0\}$, where
		      $\mathfrak{d}(P_{r,s,A})= \sqrt{\frac{1}{2}r^2 + \frac{n-1}{n}s^2}$ is the diameter of $P_{r,s,A}$ following the notation in \cite{LX24}.
	\end{enumerate}
\end{theorem}

\begin{proof}
	(1) $\Longrightarrow$ (2): Take $\varepsilon>0$, $A\in \mathbf{U}(n)$ and $\xi \in E_a \sasyugo \{0\}$.
	We have
	\[
		c \leq \frac{\langle \Theta_{h}(Ae_i, \overline{Ae_i}) \xi, \xi \rangle_{h(a)}}{|\xi|^2_{h(a)}} \leq d,
	\]
	which means that
	\[
		-\frac{c}{2}r^2 \geq -\frac{r^2}{2} \frac{\langle \Theta_{h}(Ae_1, \overline{Ae_1}) \xi, \xi \rangle_{h(a)}}{|\xi|^2_{h(a)}} \geq -\frac{d}{2}r^2, \hspace{3mm} -\frac{c}{n}s^2 \geq -\frac{s^2}{n} \frac{\langle \Theta_{h}(Ae_j, \overline{Ae_j}) \xi, \xi \rangle_{h(a)}}{|\xi|^2_{h(a)}} \geq -\frac{d}{n}s^2.
	\]
	Thanks to Theorem \ref{mainthm:Griffithsexpansion}, we can take $\rho_\varepsilon > 0$ such that for any $r,s>0$ with $r^2+s^2<\rho_\varepsilon^2$, $A\in \mathbf{U}(n)$ and $\xi \in E_a \sasyugo \{0\}$, we have
	\[
		1 - (d+\varepsilon)\mathfrak{d}(P_{r,s,A})^2 \leq L_h(a,r,s,A,\xi) \leq 1 - (c-\varepsilon)\mathfrak{d}(P_{r,s,A})^2
	\]
	(just compare $\mathfrak{d}(P_{r,s,A})^2$ with the $O((r^2+s^2)^2)$ term).
	Note that we can take such $\rho_\varepsilon$ since the estimates obtained in Theorem \ref{mainthm:Griffithsexpansion} are uniform in $A\in \mathbf{U}(n)$ and $\xi$.

	(2) $\Longrightarrow$ (1):
	To prove the inequalities in (1), we take an arbitrary $v\in \C^n$ and $\xi \in E_a \sasyugo \{0\}$.
	We can also assume that $v$ is a unit vector.
	Then we take a unitary matrix $A$ satisfying $Ae_1 = v$.
	For $\varepsilon>0$, we take $\rho_\varepsilon > 0$ such that the inequalities in (2) hold for $r, s>0$ with $r^2+s^2 < \rho_\varepsilon^2$.
	For sufficiently small $\eta>0$, we consider $s = \eta r$ and $P_{r,\eta r,A}$.
	Using the same argument as above, we can get the following inequalities
	\begin{align*}
		(c-\varepsilon)\mathfrak{d}(P_{r,\eta r,A})^2  \leq {} & {} \frac{r^2}{2} \frac{\langle \Theta_h(Ae_1, \overline{Ae_1})\xi, \xi\rangle_{h(a)}}{|\xi|^2_{h(a)}}                                          \\
		                                                       & + \frac{(\eta r)^2}{n} \sum_{j=2}^n \frac{\langle \Theta_h(Ae_j, \overline{Ae_j})\xi, \xi \rangle_{h(a)}}{|\xi|^2_{h(a)}} + O(r^4(1+\eta^2)^2) \\
		\leq                                               {}  & {} (d+\varepsilon)\mathfrak{d}(P_{r,\eta r,A})^2,
	\end{align*}
	which implies that
	\begin{align*}
		(c-\varepsilon)\left(\frac{1}{2} + \frac{n-1}{n}\eta^2 \right)  \leq {} & {} \frac{1}{2} \frac{\langle \Theta_h(Ae_1, \overline{Ae_1})\xi, \xi\rangle_{h(a)}}{|\xi|^2_{h(a)}}                                                    \\
		                                                                        & + \frac{\eta^2}{n} \sum_{j=2}^n \frac{\langle \Theta_h(Ae_j, \overline{Ae_j})\xi, \xi \rangle_{h(a)}}{|\xi|^2_{h(a)}} + \frac{O(r^4(1+\eta^2)^2)}{r^2} \\
		\leq                                               {}                   & {} (d+\varepsilon)\left(\frac{1}{2} + \frac{n-1}{n}\eta^2 \right).
	\end{align*}
	Taking the limit as $r \to 0$, we obtain
	\begin{align*}
		(c-\varepsilon)\left(\frac{1}{2} + \frac{n-1}{n}\eta^2 \right) & \leq \frac{1}{2} \frac{\langle \Theta_h(Ae_1, \overline{Ae_1})\xi, \xi\rangle_{h(a)}}{|\xi|^2_{h(a)}}
		+ \frac{\eta^2}{n} \sum_{j=2}^n \frac{\langle \Theta_h(Ae_j, \overline{Ae_j})\xi, \xi \rangle_{h(a)}}{|\xi|^2_{h(a)}}                                                               \\
		                                                               & \leq                                               (d+\varepsilon)\left(\frac{1}{2} + \frac{n-1}{n}\eta^2 \right).
	\end{align*}
	Taking the limit $\eta \to 0$ as well, we can derive from the above inequalities that
	\begin{equation*}
		(c-\varepsilon) \frac{1}{2} \leq \frac{1}{2} \frac{\langle \Theta_h(Ae_1, \overline{Ae_1})\xi, \xi\rangle_{h(a)}}{|\xi|^2_{h(a)}} \leq (d+\varepsilon) \frac{1}{2}.
	\end{equation*}
	Since $\varepsilon>0$ is arbitrary, we can get the desired inequalities.
\end{proof}

Strictly speaking, as the above proof suggests, we have equivalences between $c\omega \otimes  \Id_{E} \leq_{\mathrm{Grif.}}  \ai\Theta_{h}$ and $L_h(a,r,s,A,\xi) \leq 1 - (c-\varepsilon)\mathfrak{d}(P_{r,s,A})^2$ and between $ \ai\Theta_{h} \leq_{\mathrm{Grif.}} d\omega \otimes \Id_{E}$ and $1 - (d+\varepsilon)\mathfrak{d}(P_{r,s,A})^2 \leq L_h(a,r,s,A,\xi)$.
We remark that in \cite{Ina26}, \cite{LX24}, the authors only considered the equivalence between $c\omega \otimes  \Id_{E} \leq_{\mathrm{Grif.}}  \ai\Theta_{h}$ and $L_h(a,r,s,A,\xi) \leq 1 - (c-\varepsilon)\mathfrak{d}(P_{r,s,A})^2$, which means that there is a correlation between how sharp the $L^2$-extension is and how positive the curvature is.
Here we also prove in a more organized way that there is an equivalence between how non-sharp the $L^2$-extension is and how negative the curvature is.

As another application, combining the result in \cite{Ber11} and Theorem \ref{mainthm:Griffithsexpansion}, we get the expansion formula for a metric on a direct image sheaf.
Let $f:X\to B$ be a smooth proper K\"ahler fibration over a domain $B \subset \C$ and $(L, h=e^{-\varphi})$ be a smooth Hermitian line bundle over $X$ with $\ai\Theta_h \geq 0$.
We also assume that its restriction to every fiber is strictly positive: $\ai\Theta_h|_{X_b} > 0$ for all $b \in B$.
Set
\[
	E=f_*(K_{X/B}\otimes L)
\]
and equip $E$ with the natural $L^2$-metric $H$ defined by fiber integration with $h$.
Berndtsson's formula \cite[Theorem 1.2]{Ber11} states that the curvature of $H$ at $b\in B$ is given by
\[
	\langle \Theta_H u, u\rangle = \int_{X_b} c(\varphi) |u|^2 e^{-\varphi} + \langle (\square' +1)^{-1}\eta, \eta\rangle,
\]
where $u\in E_b \cong H^0(X_b, K_{X_b} \otimes L|_{X_b})$.
Here, we do not dive into the notation and the details since this is not the main topic of this paper.
If you are interested, please refer to \cite{Ber11} for more details.
Theorem \ref{mainthm:Griffithsexpansion} gives the following expansion formula.

\begin{corollary}\label{cor:berndtsson-directimage}
	Under the above setting, we obtain for $b\in B$ and $u\in E_b\sasyugo \{ 0\}$
	\begin{equation*}
		L_H(b,r,u) = 1 - \frac{r^2}{2 \| u\|^2_{H_b}} \left\lbrace \int_{X_b} c(\varphi) |u|^2 e^{-\varphi} + \langle (\square' +1)^{-1}\eta, \eta\rangle \right\rbrace + O (r^4).
	\end{equation*}
\end{corollary}

What we want to emphasize here is not the novelty of the result, which is just a combination of Berndtsson’s result and ours, but the fact that this formula shows, in a clear and quantitative way, the correlation between the sharpness of the $L^2$-extension, the positivity of the curvature, and the infinitesimal variation of the fibration, as measured by the Kodaira--Spencer term.

\section{$q$-positivity and flatness}\label{sec:qpositivity-flatness}
\subsection{$q$-positivity}
In this subsection, we study $q$-positivity in terms of the $L^2$-extension.
For an integer $1\leq q \leq n$, we consider the domain $P^{(q)}_{r,s,A} := A(\B^q_r\times \B^{n-q}_s)$ for $r, s>0$, so that $P^{(1)}_{r,s,A} = P_{r,s,A}$.

\begin{definition}\label{def:l2indexq}
	Let $\pi:E\to \Omega$ be a holomorphic vector bundle over $\Omega\subset \C^n$ and $h$ be a smooth Hermitian metric on $E$.
	For an integer $1\leq q \leq n$, we set $\Omega_{P^{(q)}}= \{ (a,r,s,A)\in \Omega \times \R_{>0}\times \R_{>0}\times\mathbf{U}(n) \mid a + P^{(q)}_{r,s,A}\subset \Omega \}$.
	We define {\it the $q$-$L^2$-extension index} $L^{(q)}_h$ of $h$ on $\Omega_{P^{(q)}}\times_\Omega E\sasyugo \{0\}$ by
	$$
		L^{(q)}_h(a,r,s,A,\xi)=\inf\left\{ \frac{\int_{a+P^{(q)}_{r,s,A}}|\sigma|^2_h}{|P^{(q)}_{r,s,A}| |\xi|^2_{h(a)}} ~\middle|~ \sigma\in A^2(a+P^{(q)}_{r,s,A},E; h), \sigma(a)=\xi \right\}.
	$$
\end{definition}

By definition, $L^{(1)}_h = L_h$. Using this notion, we obtain the following expansion formula, which is a generalization of Theorem \ref{mainthm:Griffithsexpansion}.

\begin{theorem}\label{thm:qexpansion}
	For fixed $a\in \Omega$ and $\xi \in E_a\sasyugo \{0\}$, we have the following expansion
	\begin{align*}
		L^{(q)}_{h}(a,r,s,A, \xi) = 1 & - \frac{r^2}{q+1} \sum_{j=1}^{q} \frac{\langle \Theta_h(Ae_j, \overline{Ae_j})\xi, \xi\rangle_{h(a)}}{|\xi|^2_{h(a)}}                                 \\
		                              & - \frac{s^2}{n-q+1} \sum_{j=q+1}^n \frac{\langle \Theta_h(Ae_j, \overline{Ae_j})\xi, \xi \rangle_{h(a)}}{|\xi|^2_{h(a)}} + O\left((r^2+s^2)^2\right).
	\end{align*}
	When $q=n$, we use the convention $\B_s^0=\{0\}$. In this case $P^{(n)}_{r,s,A}=\B_r^n$, the second sum is empty, and the remainder term can be written as $O(r^4)$.
\end{theorem}

\begin{proof}\label{proof:qexpansion}
	Since the proof is almost the same as that of Theorem \ref{mainthm:Griffithsexpansion}, we omit it.
	We only remark that the coefficients are again the mean values $\mathcal{V}_{e_j}/|P^{(q)}_{r,s,I}|$ of $|w_j|^2$ over $P^{(q)}_{r,s,I} = \B^q_r\times \B^{n-q}_s$, which are computed as follows:
	\begin{align*}
		\frac{\mathcal{V}_{e_j}}{|P^{(q)}_{r,s,I}|} & = \frac{1}{|\B^q_r|}\int_{\B^q_r} |w_j|^2 = \frac{r^2}{q+1} ~ (1\leq j \leq q),             \\
		\frac{\mathcal{V}_{e_j}}{|P^{(q)}_{r,s,I}|} & = \frac{1}{|\B^{n-q}_s|}\int_{\B^{n-q}_s} |w_j|^2 = \frac{s^2}{n-q+1} ~ (q+1\leq j \leq n).
	\end{align*}
\end{proof}

Note that the above estimates are also uniform in $A$ and $\xi$.
Using this theorem, we can give a characterization of partial positivity. In the line bundle case, this gives eigenvalue-sum conditions arising in uniform partial positivity.
Here we say $\varphi$ or $e^{-\varphi}$ is uniformly $(q-1)$-(semi)positive if the sum of any $q$ eigenvalues of $\ai\ddbar\varphi$, counted with multiplicity, is (semi)positive (see \cite{AG62}, \cite[Definition 2.1]{Yan19}).

\begin{theorem}\label{thm:qcharacterization}
	Let $\varphi$ be a smooth function on $\Omega$, $1\leq q\leq n$, and $c, d\in \R$ with $c\leq d$ and fix $a\in\Omega$.
	Then the following conditions are equivalent:
	\begin{enumerate}
		\item the sum of any $m$ eigenvalues of $\ai\ddbar\varphi(a)$ is greater than or equal to $cm$ and less than or equal to $dm$, where $m=\min\{q, n-q\}$ when $1\leq q \leq n-1$ and $m=n$ when $q=n$.
		\item for any $\varepsilon>0$, there exists $\rho_\varepsilon>0$ such that
		      $$
			      1-(d+\varepsilon)\mathfrak{d}^{(q)}(P^{(q)}_{r,s,A})^2 \leq L^{(q)}_{e^{-\varphi}}(a,r,s,A) \leq 1-(c-\varepsilon)\mathfrak{d}^{(q)}(P^{(q)}_{r,s,A})^2
		      $$
		      for any $r, s>0$ with $r^2+s^2<\rho_\varepsilon^2$, $A\in \mathbf{U}(n)$, where $\mathfrak{d}^{(q)}(P^{(q)}_{r,s,A}) = \sqrt{\frac{q}{q+1}r^2+\frac{n-q}{n-q+1}s^2}$ is the $q$-diameter of $P^{(q)}_{r,s,A}$.
	\end{enumerate}
\end{theorem}

Here we note that we do not need to specify $\xi\in E$ since we consider the line-bundle case.

\begin{proof}\label{proof:qcharacterization}
	(1) $\Longrightarrow$ (2):
	Take $\varepsilon>0, A\in \mathbf{U}(n)$.
	Assume first that $1\leq q\leq n-1$.
	From Lemma \ref{lem:subsetsum} below, we can say that the sum of any $q$ eigenvalues lies in $[cq,dq]$, and the sum of any $n-q$ eigenvalues lies in $[c(n-q),d(n-q)]$.
	By Ky Fan's principle (\cite{Fan49}), we obtain
	\[
		cq \leq \sum_{i=1}^{q} (\ddbar \varphi)(a)(Ae_i, \overline{Ae_i}) \leq dq,
	\]
	\[
		c(n-q) \leq \sum_{j=q+1}^{n} (\ddbar \varphi)(a)(Ae_j, \overline{Ae_j}) \leq d(n-q),
	\]
	which give
	\[
		-\frac{cq}{q+1}r^2 \geq - \frac{r^2}{q+1} \sum_{i=1}^{q} (\ddbar \varphi)(a)(Ae_i, \overline{Ae_i})\geq -\frac{dq}{q+1} r^2,
	\]
	\[
		-\frac{c(n-q)}{n-q+1}s^2 \geq - \frac{s^2}{n-q+1} \sum_{j=q+1}^{n} (\ddbar \varphi)(a)(Ae_j, \overline{Ae_j})\geq -\frac{d(n-q)}{n-q+1} s^2.
	\]
	Thanks to Theorem \ref{thm:qexpansion} and the same argument as in the proof of Theorem \ref{thm:equivalence}, we can get the conclusion of (2).
	When $q=n$, condition (1) gives $cn\leq\sum_{i=1}^n(\partial\bar\partial\varphi)(Ae_i,\overline{Ae_i})(a)\leq dn$. The conclusion follows from the $O(r^4)$-expansion in Theorem \ref{thm:qexpansion}.

	(2) $\Longrightarrow$ (1):
	Let $\varepsilon>0$ and $\lambda_1 \leq \cdots \leq \lambda_n$ be the eigenvalues of $\ai\ddbar\varphi(a)$.
	Choose an arbitrary subset $I=\{ i_1, \ldots, i_q\}$ of $\{1, \ldots, n\}$, write $\{ i_{q+1}, \ldots, i_n\} = \{1, \ldots, n\}\sasyugo I$, and let $A$ be a unitary matrix such that $Ae_\ell$ is an eigenvector corresponding to $\lambda_{i_\ell}$ for $\ell = 1, \ldots, n$.
	Then we have
	\begin{align*}
		1 -(d+\varepsilon)\mathfrak{d}^{(q)}(P^{(q)}_{r,s,A})^2
		 & \leq 1 - \frac{r^2}{q+1}\sum_{\ell=1}^q \lambda_{i_\ell} - \frac{s^2}{n-q+1}\sum_{\ell=q+1}^n \lambda_{i_\ell} + O((r^2+s^2)^2) \\
		 & \leq 1-(c-\varepsilon)\mathfrak{d}^{(q)}(P^{(q)}_{r,s,A})^2
	\end{align*}
	thanks to Theorem \ref{thm:qexpansion}.
	For sufficiently small $\eta>0$, we consider $s=\eta r$ and $P^{(q)}_{r,\eta r,A}$.
	Then we get
	\begin{align*}
		(c-\varepsilon)\left(\frac{q}{q+1}r^2+\frac{n-q}{n-q+1}(\eta r)^2 \right)
		 & \leq \frac{r^2}{q+1}\sum_{\ell=1}^q \lambda_{i_\ell} + \frac{\eta^2r^2}{n-q+1}\sum_{\ell=q+1}^n \lambda_{i_\ell} + O(r^4(1+\eta^2)^2) \\
		 & \leq (d+\varepsilon)\left(\frac{q}{q+1}r^2+\frac{n-q}{n-q+1}(\eta r)^2 \right).
	\end{align*}
	Repeating the same argument as in the proof of Theorem \ref{thm:equivalence}, we can obtain
	\[
		cq \leq \sum_{\ell=1}^q \lambda_{i_\ell} \leq dq.
	\]

	If $1\leq q\leq n-1$, interchanging the roles of $r$ and $s$ and repeating the same process with $r=\eta s$, we can also say that the sum of any $n-q$ eigenvalues of $\ai\ddbar\varphi(a)$, counted with multiplicity, is greater than or equal to $c(n-q)$ and less than or equal to $d(n-q)$.

	When $q=n$, we can apply the same argument and get the conclusion of (1).
\end{proof}

\begin{lemma}\label{lem:subsetsum}
	Let $\lambda_1, \ldots, \lambda_n \in \R$, $c \leq d$, and $1 \leq m \leq M \leq n$.
	If the sum of any $m$ distinct $\lambda_i$'s lies in $[cm, dm]$,
	then the sum of any $M$ distinct $\lambda_i$'s lies in $[cM, dM]$.
\end{lemma}

\begin{proof}
	For $T \subset \{1, \ldots, n\}$ with $|T| = M$, summing over all
	$S \subset T$ with $|S| = m$ gives
	$\sum_{S} \sum_{i \in S} \lambda_i = \binom{M-1}{m-1} \sum_{i \in T} \lambda_i$,
	since each $i \in T$ belongs to exactly $\binom{M-1}{m-1}$ such subsets.
	As each inner sum lies in $[cm, dm]$ and
	$\binom{M}{m} m = \binom{M-1}{m-1} M$, the claim follows.
\end{proof}

\begin{remark}\label{rem:qvectorbundle}
	Applying the same argument as in the proof of Theorem \ref{thm:qcharacterization}, we obtain the following type of result
	$$
		cm\leq\sum_{i=1}^m\frac{\langle\Theta_h(Ae_i,\overline{Ae_i})\xi,\xi\rangle_{h(a)}}{|\xi|_{h(a)}^2}\leq dm
	$$
	for every $\xi\in E_a\setminus\{0\}$ and $A\in\mathbf U(n)$ if  $1-(d+\varepsilon)\mathfrak{d}^{(q)}(P^{(q)}_{r,s,A})^2 \leq L^{(q)}_{h}(a,r,s,A, \xi) \leq 1-(c-\varepsilon)\mathfrak{d}^{(q)}(P^{(q)}_{r,s,A})^2$, where $m=\min\{q,n-q\}$ for $1\leq q\leq n-1$, and $m=n$ for $q=n$. The same argument also gives
	$$
		cq\leq\sum_{i=1}^q\frac{\langle\Theta_h(Ae_i,\overline{Ae_i})\xi,\xi\rangle_{h(a)}}{|\xi|_{h(a)}^2}\leq dq,
	$$
	which will be used in the next subsection.
	Here, however, we omit the details.
\end{remark}

For the same reason as in the proof of Theorem \ref{thm:equivalence}, the condition
$1-(d+\varepsilon)\mathfrak{d}^{(q)}(P^{(q)}_{r,s,A})^2 \leq L^{(q)}_{e^{-\varphi}}(a,r,s,A)$
corresponds to the sum of any $m$ eigenvalues of $\ai\ddbar\varphi(a)$, counted with multiplicity, being less than or equal to $dm$, and the condition
$L^{(q)}_{e^{-\varphi}}(a,r,s,A) \leq 1-(c-\varepsilon)\mathfrak{d}^{(q)}(P^{(q)}_{r,s,A})^2$
corresponds to the sum of any $m$ eigenvalues of $\ai\ddbar\varphi(a)$, counted with multiplicity, being greater than or equal to $cm$.
Therefore, taking $c=0$ and applying the above argument,
we obtain the following $(m-1)$-positivity version of the minimal/optimal $L^2$-extension property introduced in \cite{HPS18},\cite{DNW21},\cite{DNWZ23}.

\begin{theorem}\label{thm:optimalqpositivity}
	Keep the same setting as before.
	Assume that $L^{(q)}_{e^{-\varphi}}(a,r,s,A)\leq 1$ for any $(a,r,s,A)\in\Omega_{P^{(q)}}$.
	Then $\varphi$ is uniformly $(m-1)$-semipositive, where $m=\min\{q,n-q\}$ when $1\leq q\leq n-1$ and $m=n$ when $q=n$.
\end{theorem}

If we set $q=1$, this recovers the minimal/optimal $L^2$-extension property, which states that if $L_{e^{-\varphi}}(a,r,s,A)\leq 1$, then $\varphi$ is plurisubharmonic.

\subsection{Flatness}

Many researchers have studied the relationship between the flatness of a Hermitian metric and the equality part of the optimal $L^2$-extension property (see \cite{Ina23}, \cite{Ina26}, \cite{LX24}).
Those studies show that the metric being flat is equivalent to the $L^2$-extension index being equal to $1$.
In this context, it is natural and interesting to ask whether we can characterize the flatness of a Hermitian metric in terms of the $q$-$L^2$-extension index.
In this subsection, we give an answer to this question.

\begin{theorem}\label{thm:flat-q}
	Keep the same setting as in Theorem \ref{thm:qexpansion} and assume $1\leq q\leq n-1$.
	Then the following are equivalent:
	\begin{enumerate}[font=\upshape]
		\item $\ai\Theta_h\equiv 0$.
		\item $L^{(q)}_h \equiv 1$.
	\end{enumerate}
\end{theorem}

\begin{proof}
	(1) $\Longrightarrow$ (2): It follows from the argument in \cite[Theorem 5.1]{LX24}.

	(2) $\Longrightarrow$ (1): Taking $c=d=0$ and considering Theorem \ref{thm:qcharacterization} and Remark \ref{rem:qvectorbundle}, we get for $\xi\in E_a\sasyugo \{ 0\}$ and $A\in \mathbf{U}(n)$
	\[
		\sum_{j=1}^{q} H_\xi(Ae_j) = 0,
	\]
	where $H_\xi(v) = \frac{\langle \Theta_h(v, \overline{v})\xi, \xi\rangle_{h(a)}}{|\xi|^2_{h(a)}}$.
	Take arbitrary unit vectors $u, w$.
	If $u$ and $w$ are linearly dependent, $H_\xi(u)=H_\xi(w)$.
	If they are linearly independent, we can choose orthonormal vectors $(v_2, \ldots, v_q)$ that are orthogonal to $u$ and $w$ since $q-1\leq n-2$.
	We take unitary matrices $A_u, A_w$ satisfying
	\[
		A_u(e_1)=u, A_u(e_j)=v_j,
	\]
	\[
		A_w(e_1)=w, A_w(e_j)=v_j
	\]
	for $2\leq j\leq q$.
	Then we have
	\[
		H_\xi(u) + \sum_{j=2}^{q} H_\xi(v_j) = 0,
	\]
	\[
		H_\xi(w) + \sum_{j=2}^{q} H_\xi(v_j) = 0.
	\]
	Hence $H_\xi(u) = H_\xi(w)$.
	Since $u$ and $w$ are arbitrary, we conclude that $H_\xi(v)$ is constant for any unit vector $v$, which also means $H_\xi(v)=0$.
	Since $H_\xi(v)=0$ for every $v$ and every $\xi$, polarization in $\xi$ first gives $\Theta_h(v, \bar{v})=0$.
	Then polarization in $v$ again gives the conclusion (1).
\end{proof}

As the proof suggests, the assumption $1\leq q\leq n-1$ plays an essential role.

\appendix
\section{}
\label{appendix}

\begin{lemma}\label{lem:frame}
	Let $\Omega$ be a domain in $\C^n$, $\pi:E\to \Omega$ be a holomorphic vector bundle of rank $r$ over $\Omega$, and $h$ be a smooth Hermitian metric on $E$.
	For every $a\in \Omega$ and every coordinate system $(z_j)_j$ centered at $a$, there exists a holomorphic frame $\{e_\alpha\}_{\alpha=1}^r$ of $E$ on a neighborhood of $a$ such that
	\begin{align*}
		\langle e_\lambda, e_\mu \rangle_{h} & = \delta_{\lambda \mu} - \sum_{j,k} c_{j\bar{k}\lambda\bar{\mu}} z_j\bar{z}_k + \sum_{j_1, j_2, k} q^{j_1j_2\bar{k}}_{\lambda\bar{\mu}}z_{j_1}z_{j_2}\bar{z}_k + \sum_{j, k_1, k_2} r^{j\bar{k}_1\bar{k}_2}_{\lambda\bar{\mu}} z_j \bar{z}_{k_1} \bar{z}_{k_2} + O(|z|^4),
	\end{align*}
	where $(c_{j\bar{k}\lambda\bar{\mu}})$ are the coefficients of the Chern curvature tensor $\Theta_h$ at $a$ and $q^{j_1j_2\bar{k}}_{\mu\bar{\lambda}} = \overline{r^{k\bar{j}_1\bar{j}_2}_{\lambda\bar{\mu}}}$.
\end{lemma}

\begin{proof}
	By \cite[Proposition 12.10, Chapter V]{Dem-agbook}, we can take a holomorphic frame $\{e_\lambda\}_{\lambda=1}^r$ of $E$ on a neighborhood of $a$ such that
	\begin{align*}
		\langle e_\lambda, e_\mu \rangle_{h} = \delta_{\lambda \mu} - \sum_{j,k} c_{j\bar{k}\lambda\bar{\mu}} z_j\bar{z}_k & + \sum_{j_1, j_2, j_3} p^{j_1 j_2 j_3}_{\lambda\bar{\mu}}z_{j_1}z_{j_2}z_{j_3} + \sum_{j_1, j_2, k} q^{j_1j_2\bar{k}}_{\lambda\bar{\mu}}z_{j_1}z_{j_2}\bar{z}_k \\                                                                                                                               &+ \sum_{j, k_1, k_2} r^{j\bar{k}_1\bar{k}_2}_{\lambda\bar{\mu}} z_j \bar{z}_{k_1} \bar{z}_{k_2} + \sum_{k_1, k_2, k_3} s^{\bar{k}_1\bar{k}_2\bar{k}_3}_{\lambda\bar{\mu}} \bar{z}_{k_1}\bar{z}_{k_2}\bar{z}_{k_3} + O(|z|^4).
	\end{align*}
	Hermitian symmetry $\langle e_\lambda, e_\mu \rangle_h = \overline{\langle e_\mu, e_\lambda \rangle}_h$ implies that
	\[
		p^{j_1j_2j_3}_{\mu\bar{\lambda}} = \overline{s^{\bar{j}_1\bar{j}_2\bar{j}_3}_{\lambda\bar{\mu}}}, \hspace{5mm} q^{j_1j_2\bar{k}}_{\mu\bar{\lambda}} = \overline{r^{k\bar{j}_1\bar{j}_2}_{\lambda\bar{\mu}}}.
	\]
	We define a new frame $e'_\lambda$ by
	$$
		e'_\lambda = e_\lambda - \sum_{|\gamma| = 3}\sum_{k} z^\gamma p^\gamma_{\lambda\bar{k}}e_k.
	$$
	Here the multi-index sum $\sum_{|\gamma| = 3}$ runs over $\gamma = (\gamma_1, \ldots, \gamma_n)\in \Z^n_{\geq 0}$ with $|\gamma| := \gamma_1 + \cdots + \gamma_n = 3$, $z^\gamma = z_1^{\gamma_1}\cdots z_n^{\gamma_n}$, and $p^\gamma_{\lambda\bar{\mu}}$ denotes the coefficient of $z^\gamma$.
	Then we have
	\begin{align*}
		\langle e'_\lambda, e'_\mu \rangle_{h}  = {} & {} \langle e_\lambda, e_\mu \rangle_{h} - \langle \sum_{|\gamma|=3}\sum_k z^\gamma p^\gamma_{\lambda\bar{k}}e_k, e_\mu  \rangle_h      - \langle e_\lambda, \sum_{|\delta|=3}\sum_\ell z^\delta p^\delta_{\mu\bar{\ell}}e_\ell \rangle_h \\
		                                             & + \langle \sum_{|\gamma|=3}\sum_k z^\gamma p^\gamma_{\lambda\bar{k}}e_k,  \sum_{|\delta|=3}\sum_\ell z^\delta p^\delta_{\mu\bar{\ell}}e_\ell \rangle_h.
	\end{align*}
	The second term on the right-hand side is equal to
	\begin{align*}
		\langle \sum_{|\gamma|=3}\sum_k z^\gamma p^\gamma_{\lambda\bar{k}}e_k, e_\mu \rangle_h & = \sum_{|\gamma|=3}\sum_k z^\gamma p^\gamma_{\lambda\bar{k}} \langle e_k, e_\mu \rangle_h          \\
		                                                                                       & =  \sum_{|\gamma|=3}\sum_k \left( z^\gamma p^\gamma_{\lambda\bar{k}}\delta_{k\mu}+ O(|z|^5)\right) \\
		                                                                                       & = \sum_{|\gamma|=3} z^\gamma p^\gamma_{\lambda\bar{\mu}} + O(|z|^5).
	\end{align*}
	Similarly, we get
	$$
		\langle e_\lambda, \sum_{|\delta|=3}\sum_\ell z^\delta p^\delta_{\mu\bar{\ell}}e_\ell \rangle_h =\sum_{|\delta|=3} \bar{z}^\delta \overline{p^{\delta}_{\mu\bar{\lambda}}} + O(|z|^5) =\sum_{|\delta|=3} \bar{z}^\delta s^{\bar{\delta}}_{\lambda\bar{\mu}} + O(|z|^5).
	$$
	It follows that
	\begin{align*}
		\langle e'_\lambda, e'_\mu \rangle_{h} & = \delta_{\lambda \mu} - \sum_{j,k} c_{j\bar{k}\lambda\bar{\mu}} z_j\bar{z}_k + \sum_{j_1, j_2, k} q^{j_1j_2\bar{k}}_{\lambda\bar{\mu}}z_{j_1}z_{j_2}\bar{z}_k + \sum_{j, k_1, k_2} r^{j\bar{k}_1\bar{k}_2}_{\lambda\bar{\mu}} z_j \bar{z}_{k_1} \bar{z}_{k_2} + O(|z|^4).
	\end{align*}
\end{proof}








\begin{thebibliography}{99}
	\bibitem[AG62]{AG62}	A. Andreotti and H. Grauert, \emph{Th\'eor\`eme de finitude pour la cohomologie des espaces complexes}, Bull. Soc. Math. France \textbf{90} (1962), 193--259.
	\bibitem[Ber11]{Ber11} B.~Berndtsson, \emph{Strict and nonstrict positivity of direct image bundles}, Math. Z. \textbf{269} (2011), 1201--1218.
	\bibitem[B\l o13]{Blo13} Z.~B\l ocki, \emph{Suita conjecture and the Ohsawa-Takegoshi extension theorem}, Invent. Math. \textbf{193} (2013), 149--158.
	\bibitem[Dem-book]{Dem-agbook} J.-P. Demailly, \emph{Complex analytic and differential geometry},
	\newline
	\url{http://www-fourier.ujf-grenoble.fr/~demailly/manuscripts/agbook.pdf}.
	\bibitem[DNW21]{DNW21} F.~Deng, J.~Ning, and Z.~Wang, \emph{Characterizations of plurisubharmonic functions}, Sci. China Math. \textbf{64} (2021), 1959--1970.
	\bibitem[DNWZ23]{DNWZ23} F.~Deng, J.~Ning, Z.~Wang, and X.~Zhou, \emph{Positivity of holomorphic vector bundles in terms of $L^p$-estimates for $\overline{\partial}$}, Math. Ann. \textbf{385} (2023), 575--607. 
	\bibitem[Fan49]{Fan49} K.~Fan, \emph{On a theorem of Weyl concerning eigenvalues of linear transformations. I}, Proc. Nat. Acad. Sci. U. S. A. \textbf{35} (1949), 652--655.
	\bibitem[GZ15]{GZ15} Q.~Guan and X.~Zhou, \emph{A solution of an $L^2$ extension problem with an optimal estimate and applications}, Ann. of Math. (2) \textbf{181} (2015), no.~3, 1139--1208.
	\bibitem[HPS18]{HPS18} C.~Hacon, M.~Popa, and C.~Schnell, \emph{Algebraic fiber spaces over Abelian varieties: around a recent theorem by Cao and Paun}, Contemp. Math. \textbf{712} (2018), 143--195.
	\bibitem[Hos19]{Hos19} G.~Hosono, \emph{On sharper estimates of Ohsawa--Takegoshi $L^2$-extension theorem}, J. Math. Soc. Japan \textbf{71} (2019), no.~3, 909--914.
	\bibitem[Ina23]{Ina23} T.~Inayama, \emph{A note on characterizing pluriharmonic functions via the Ohsawa--Takegoshi extension theorem}, J. Math. Sci. Univ. Tokyo \textbf{30} (2023), no.~3, 365--369.
	\bibitem[Ina26]{Ina26} T.~Inayama, \emph{$L^2$-extension indices, sharper estimates and curvature positivity}, Annales de l'Institut Fourier \textbf{76} (2026), no.~4, 1401--1429.
	\bibitem[Kik23]{Kik22} S.~Kikuchi, \emph{On sharper estimates of Ohsawa--Takegoshi $L^2$-extension theorem in higher dimensional case}, manuscripta math. \textbf{170} (2023), 453--469.
	\bibitem[LX24]{LX24} Z.~Liu and W.~Xu, \emph{Characterizations of Griffiths positivity, pluriharmonicity and flatness}, J. Funct. Anal. \textbf{287} (2024), no. 7, 110532.
	\bibitem[OT87]{OT87} T.~Ohsawa and K.~Takegoshi, \emph{On the extension of $L^2$ holomorphic functions}, Math. Z. \textbf{195}, (1987), no.~2, 197--204.
	\bibitem[Yan19]{Yan19} X. Yang, \emph{A partial converse to the Andreotti-Grauert theorem}, Compos. Math. \textbf{155} (2019), 89--99.
\end{thebibliography}
\end{document}